\documentclass[10pt,twoside,a4paper]{article}
\usepackage[english,catalan,galician,activeacute]{babel}
\usepackage[T1]{fontenc}
\usepackage{inputenc}
\usepackage{amsmath,bm,mathrsfs}
\usepackage{amssymb}
\usepackage{amsxtra}
\usepackage{amsthm}
\usepackage{epic}
\usepackage{eepic} 
\usepackage{fancybox}
\usepackage{boxedminipage}
\usepackage{calc}
\usepackage{ifthen}
\usepackage{paralist}
\usepackage{exscale,relsize}
\usepackage{pgf}
\usepackage[rm,sc,small,center]{titlesec}
\titlelabel{\thetitle.\enspace}
\usepackage{tikz}
\usetikzlibrary{arrows,automata,backgrounds,calendar,snakes}
\usepackage{inputenc}
\usepackage{color}

\newtheorem{theorem}{Theorem}[section]
\newtheorem{lemma}{Lemma}[section]
\newtheorem{remark}{Remark}[section]

\newcommand{\R}{\mathbb{R}}

\newcommand{\be}{\begin{eqnarray}}
\newcommand{\ee}{\end{eqnarray}}
\newtheorem{prop}{Proposition}

\begin{document}
\selectlanguage{english}
\thispagestyle{empty}
\numberwithin{equation}{section}
\pagestyle{myheadings}
\markboth{\small H. Freist\"uhler}{\small 
Phase Boundaries in Relativistic Korteweg Fluids}
\title{On Phase Boundaries in Relativistic Korteweg Fluids} 
\author{Heinrich Freist\"uhler} 
\date{August 21,2023}
\selectlanguage{english}
\maketitle

\begin{abstract}
This is the first of several planned papers that study the existence and local-in-time persistence  of 
phase fronts in the Lorentz invariant Euler equations for gases of van der Waals type, aiming at 
transferring earlier results of Slemrod
and of Benzoni-Gavage and collaborators to the context of the theory of relativity. While the later 
papers will examine more general admissibility criteria, this one uses and extends the author's Lorentz 
invariant formulation of Korteweg's theory of capillarity, establishing a family of phase fronts 
that have a regular heteroclinic profile with respect to the associated Euler-Korteweg equations. 
\end{abstract}

\section{Introduction}
This article studies aspects of the dynamics of compressible fluids that can assume two 
different phases. While there are situations in which it is more appropriate to view such fluids 
as mixtures of two different materials and use (e.g., a concentration as) an additional phase field 
variable, we 
here follow the classical description going back to van der Waals \cite{vdW}, 
in which it is the fluid's density 
itself that plays the r\^ole of the order parameter. For perfect fluids of van der Waals type, phase boundaries 
are 
fronts of discontinuity that separate regions supporting one of
the phases from regions filled with the other; such ideal phase boundaries may be studied from the 
point of view of the theory of weak solutions to nonlinear hyperbolic systems of conservation laws.
This thinking began with 
papers \cite{Sl83,Sl84} by Slemrod, who pointed out that a proper treatment requires 
a further relation complementing the classical Rankine-Hugoniot conditions. Slemrod also noticed that, 
differently from the case of shock waves, restricting attention for higher-order effects to viscosity
alone cannot describe the assumed internal structure 
of phase boundaries, but the inclusion 
of capillary forces can; he obtained his additonal jump condition using Korteweg's classical theory of 
capillarity \cite{Ko,TN}. 
Later, Benzoni-Gavage showed that using Slemrod's admissibility criterion 
a Lopatinski condition holds, namely neutrally if one refers to capillarity only \cite{B1} 
and uniformly if one refers to capillarity \emph{and} viscosity \cite{B2}; this finding makes
phase fronts amenable the local-in-time stability theory developed by Majda \cite{Ma} 
and M\'etivier \cite{Me}.
If, again neglecting viscosity, one actually introduces capillarity into the partial differential equations
themselves, the dynamics is governed by the Euler-Korteweg system that has been derived by Dunn and 
Serrin \cite{DS}. This system and the long-time stability of diffuse
interphase interfaces in it have been studied thoroughly by Benzoni-Gavage and collaborators \cite{B3,B4}.
\par\medskip
All these phenomena should have counterparts in the Lorentz invariant relativistic setting, and the 
purpose of the present paper is to take one step in this direction. A relativistic version of the 
Euler-Korteweg equations has been formulated in \cite{F18} both for isentropic and for general 
non-barotropic fluids. A first contribution of the present paper is here the consistent treatment of 
general barotropic fluids that covers in particular the relativistic thermobarotropic case, a class of 
fluids that specifically arise in the relativistic context where the conservation of energy-momentum 
can make sense for fluids without even the presence of a particle number density; the distinction, 
notably at the mathematical level, of this class from the (somewhat artificial) one of isentropic fluids 
was studied in paper \cite{F19} whose findings we extend here to the Euler-Korteweg setting. 
We show that both the Euler and the Euler-Korteweg equations can conveniently be expressed  in terms
of the fluid's Lichnerowicz index $f$ \cite{Li}, or its \emph{extended} version $\bar f$, respectively, 
and an associated conjugate index $\nu$. This $(\bar f,\nu)$ formulation of the Euler-Korteweg equation
enables a clean treatment of the \emph{capillary-profile problem}, i.e., the task of finding traveling
waves that are heteroclinic to boundary states in the two different phases. We solve this problem both    
for barotropic and non-barotropic fluids, thus establishing a relativistic counterpart of the picture 
Slemrod created for non-relativistic fluids. 
We expect that like in the non-relativistic setting, also \emph{visco-capillary} profiles exist --
we think of using notably the covariant viscosity developed by Temple and the author (cf.\ \cite{FT23, 
FT17} and references therein) -- and the resulting \emph{Slemrod admissibility criterion} will 
lead to validity of the uniform Lopatinski condition for the corresponding sharp interfaces. A study 
to that end is on the way \cite{FRS}. Finally, a similar transfer from the non-relativistic to the 
relativistic framework should hold for phase-field models \cite{F14, FK17, FK19}. 

Our transfer from the classical Galilei invariant setting to the Lorentz invariant context
seems of a principal interest, already as physically speaking the latter is more fundamental, so that 
any severe obstacles for such transfer would pose questions also on the non-relativistic 
models. But we also imagine concrete applications of the considerations in this paper and notably 
those in \cite{FRS}, possibly to phase boundaries in gaseous stars. 
In that regard, papers by Christodoulou  have been a major source of motivation for our attempting this 
transfer. In \cite{C95,C96} he considered a star that consists of \emph{dust} and a (relativistically) 
\emph{incompressible} phase: the pressure $\hat p(\rho)$ vanishes for energy density $\rho$ below a 
certain threshold $\rho_*$ and equals $\rho-\rho_*$ for $\rho$ above $\rho_*$, which corresponds 
to a degenerate limit of a van der Waals gas. While that limit seems particularly appropriate for 
capturing behavior in the late period of a collapsing star, considering an actual, nondegenerate, 
van der Waals type equation of state might enable the modelling of more general situations. 
A different interesting aspect of Christodoulou's work is that he studied the problem 
in the setting of \emph{general} relativity.  While the present paper remains completely
within the framework of \emph{special} relativity, an extension of the transfer begun here to the context of general
relativity seems very feasible; regarding the treatment of shock waves, we refer to \cite{FR} for 
such an extension.    

In Section 2, we set the stage by discussing basics and the $(f,\nu)$ formulation of \emph{perfect} 
barotropic fluids and its two interpretations for thermobarotropic and isentropic materials.      
Section 3 develops the Euler-Korteweg equations for barotropic fluids and the associated
$(\bar f,\nu)$ formulation. In Sections 4 and 5, we show the existence of heteroclinic solutions 
to the Euler-Korteweg equation, for general barotropic and general non-barotropic fluids, 
respectively. We end this introduction with a remark. Not surprisingly from a physics point of view, 
the general picture obtained in this paper displays an overall similarity with the one papers 
\cite{Sl83,Sl84,B1} have established in the non-relativistic setting. The interested reader will 
however notice that the qualitative difference between relativistic barotropic fluids and what are often 
called barotropic fluids in the non-relativistic context leads to a fine exchange of r\^oles between
the \emph{kinetic rule} and part of the Rankine-Hugoniot relations.  
\goodbreak
\section{Barotropic perfect fluids}
We begin with a consideration on perfect fluids in a special case which does \emph{not} admit 
phase boundaries.  
A  perfect relativistic fluid is barotropic and everywhere stable 
if there is a one-to-one relation between energy and pressure,
\be\label{barotropic}
p=\hat p(\rho),\quad \rho=\hat\rho(p) 
\quad \text{with $\hat p',\hat\rho'>0$ and }
\hat p\circ\hat\rho
=\hat\rho\circ\hat p=\text{id}_{(0,\infty)}.
\ee
Besides by $\rho$ or $p$, the local state of the fluid can be expressed via either its Lichnerowicz index 
$f$ \cite{Li} or its conjugate index $\nu$, quantities that are defined as
\be\label{fnu} 
f=\hat f(p)=\exp\int_1\frac{dp}{\hat\rho(p)+p}
\quad\text{and}\quad 
\nu=\hat\nu(\rho)=\exp\int_1\frac{d\rho}{\rho+\hat p(\rho)}.
\ee 
The functions $\pi=\hat f^{-1}$ and $r=\hat\nu^{-1}$ 
with which 
\be\label{prho}
p=\pi(f),
\quad
\rho=r(\nu)
\ee
are strictly convex, i.e., $\pi'',r''>0$, and satisfy
\be\label{Leg}
r(\nu)+\pi(f)=\nu f,
\ee
i.e., energy and pressure are Legendre conjugate, 
with 
\be\label{fnuLeg}
f=r'(\nu),
\quad
\nu=\pi'(f)
\ee
as dual variables; cf.\ \cite{F19}. 
With $U^\alpha$ the 4-velocity and $\Pi^{\alpha\beta}=g^{\alpha\beta}+U^\alpha U^\beta$ 
the projection orthogonal to $U^\alpha$, the dynamics of such fluid is described by the barotropic 
Euler equations  
\be\label{Eulerbaro}
T^{\alpha\beta}_{,\beta}
=0
\ee
where in the stress-energy tensor
\be\label{idealstress}
T^{\alpha\beta}=\rho U^{\alpha}U^{\beta}+p\Pi^{\alpha\beta},
\ee
energy and pressure may be interpreted as
\be\label{poff}
p=\pi(f),\quad \rho=f \pi'(f)-\pi(f)
\ee
or as 
\be\label{rhoofnu}
\rho=r(\nu),\quad p=\nu r'(\nu)-r(\nu).
\ee
\begin{prop}
For smooth solutions, system \eqref{Eulerbaro}, \eqref{idealstress} is equivalent to the pair consisting 
of the companion law
\be\label{nuconservation}
(\nu U^\beta)_{,\beta}=0
\ee 
and the equations of motion 
\be\label{eomf}
U^\beta\big(fU^\alpha\big)_{,\beta}+f^{,\alpha}=0.
\ee 
\end{prop}
This follows from straightforward calculations (which are contained in the proof of Theorem 2.1 below
as the special case $\kappa=0$).
\goodbreak

The quantities $\nu$ and $f$ can be interpreted in various ways. The most natural interpretation 
of $\nu$ is that of entropy $S$, in which case $f$ is the temperature $\theta$. In that case, 
in which one might call the fluid \emph{thermobarotropic}, relations \eqref{poff} and \eqref{rhoofnu}
become 
\be\label{poftheta}
p=\pi(\theta),\quad \rho=\theta \pi'(\theta)-\pi(\theta),
\ee
\be\label{rhoofS}
\rho=r(S) ,\quad p=S r'(S)-r(S),
\ee
while the PDEs \eqref{nuconservation}, \eqref{eomf}  read 
\be\label{Sconservation}
(S U^\beta)_{,\beta}=0
\ee 
\be\label{eomtheta}
U^\beta\big(\theta U^\alpha\big)_{,\beta}+\theta^{,\alpha}=0.
\ee 

In a different interpretation, $\nu$ is considered as a particle number density $n$ and
$f$ becomes the enthalpy $h=(\rho+p)/n$. In that picture, of an \emph{isentropic} massive fluid, 
relations \eqref{poff} and \eqref{rhoofnu}
read 
\be\label{pofh}
p=\pi(h),\quad \rho=h \pi'(h)-\pi(h),
\ee
\be\label{rhoofn}
\rho=r(n) ,\quad p=n r'(n)-r(n),
\ee
while the PDEs \eqref{nuconservation}, \eqref{eomf} become 
\be\label{nconservation}
(n U^\beta)_{,\beta}=0
\ee 
\be\label{eomh}
U^\beta\big(h U^\alpha\big)_{,\beta}+h^{,\alpha}=0.
\ee 
This interpretation is of limited applicability, notably in view of the fact that  
the overall description \eqref{nuconservation}, \eqref{eomf} is good only for regular 
solutions. In particular the companion law \eqref{nuconservation} 
(cf.\ \cite{Da})
does not generally hold for weak solutions; in the presence of shock waves, it must be replaced 
 by 
\be\label{nuproduction}
(\nu U^\beta)_{,\beta}>0,
\ee 
an inequality that makes sense in the case $\nu=S$, namely as the second law of thermodynamics, 
but not, physically, for $\nu=n$ (cf.\ \cite{F19}).  
\par\bigskip
We now switch from the above situation to a more general one,
so as to allow for fluids which are \emph{not} everywhere stable: 
We keep assuming that the pressure is a function of the energy, but not necessarily vice versa, 
i.e., of the two functions in \eqref{barotropic}, $\hat p$ may be non-monotone and $\hat\rho$ 
non-existent. In this situation we can 
still use ``half'' of the above description, defining $\nu$ as in \eqref{fnu}$_2$, with 
\be\label{nonconvex}
r>0\text{ and $\nu r'(\nu)>r(\nu)$, \quad but not necessarily $r''>0$ everywhere};
\ee
the range $\{\nu:r''(\nu)<0\}$ corresponds to a \emph{spinodal} region for the Euler equations
\eqref{Eulerbaro}, \eqref{rhoofnu}, where the derivative
$$
\hat p'(\rho)=\frac{\nu r''(\nu)}{r'(\nu)}, 
$$
which otherwise is the squared sound speed $c_s^2$, assumes negative values.
While formulas \eqref{fnu}$_1$, \eqref{prho}$_1$, \eqref{Leg}, \eqref{fnuLeg}$_2$, \eqref{poff} cannot generally be used now, the Lichnerowicz index 
$$
f=f(\nu)
$$
can still be defined using \eqref{fnuLeg}$_1$. 
\newpage
\noindent
This more general situation covers in particular fluids of \emph{van der Waals} type, for which 
sgn$(\hat p'(\rho))=\text{sgn}(r''(\nu))$ switches first from positive to negative and then back 
to negative, as $\rho$ (like $\nu$) varies from $0$ to $\infty$.    
\par\bigskip
\section{Barotropic Korteweg fluids}
The considerations of this section are independent  of whether the fluid we study is of van der Waals type
or not. Their purpose is to formulate, for general barotropic fluids,
the relativistic counterpart of the equations given by Korteweg \cite{Ko} for fluids with capillarity.
\par\bigskip
To do so, we modify \eqref{rhoofnu} and work with the \emph{extended energy}  
\be
\bar\rho(\nu,\partial\nu)=r(\nu)+\frac12\kappa(\nu)\nu_{,\gamma}\nu^{,\gamma}
\ee
and the \emph{extended pressure} 
\be
\bar p(\nu,\partial\nu,\partial^2\nu)
=\nu\frac{\delta\bar\rho(\nu,\partial\nu)}{\delta \nu}-\bar\rho(\nu,\partial\nu),
\ee
where
\be
\frac{\delta\bar\rho(\nu,\partial\nu)}{\delta \nu}
=
\frac{\partial\bar\rho(\nu,\partial\nu)}{\partial\nu}
-\left(\frac{\partial\bar\rho(\nu,\partial\nu)}{\partial\nu_{,\gamma}}\right)_{,\gamma},
\ee
and define -- generalizing \cite{F18} -- the associated Euler-Korteweg equations by 
\be\label{EKbaro}
\bar T^{\alpha\beta}_{,\beta}
=0
\ee
where
\be\label{fullT}
\bar T^{\alpha\beta}=\bar\rho U^{\alpha}U^{\beta}+\bar p\Pi^{\alpha\beta}+\check K^{\alpha\beta}
\ee
with
\be\label{defcheckK}
\check K^{\alpha\beta}
=
\kappa \nu^{,\alpha}\nu^{,\beta}
\ee
the non-scalar part of the (negative) \emph{relativistic Korteweg tensor} 
\be\label{defK}
K^{\alpha\beta}
=
-\nu(\kappa \nu^{,\gamma})_{,\gamma}\Pi^{\alpha\beta}
+\kappa \nu^{,\alpha}\nu^{,\beta}.
\ee
\begin{theorem}
For smooth solutions, system \eqref{EKbaro}--\eqref{defcheckK} is equivalent to the pair consisting 
of the companion law
\be\label{nuconservationagain}
(\nu U^\beta)_{,\beta}=0
\ee 
and the equations of motion 
\be\label{eombarf}
U^\beta\big(\bar fU^\alpha\big)_{,\beta}+\bar f^{,\alpha}=0,
\ee 
with the \emph{extended Lichnerowicz index} 
\be\label{barf}
\bar f(\nu,\partial\nu,\partial^2\nu)=\frac{\delta\bar\rho(\nu,\partial\nu)}{\delta \nu}.
\ee
\end{theorem}


\newpage
\section{Phase boundaries in barotropic Korteweg fluids}
We now look for traveling waves, i.e., smooth solutions $(\nu,U^0,U^1,U^2,U^3)$ to system 
\eqref{EKbaro}--\eqref{defcheckK} which are stationary in some Lorentz frame. Working in such frame
we assume w.\ l.\ o.\ g.\ that solutions have $U^2,U^3=0$ and depend only on $x^1$. 
As we are interested
in phase boundaries, we will require the waves to be heteroclinic, i.e., they converge for 
$x^1\to\pm\infty$ to end states which are different from each other.
The purpose of this section is to show the following.
\begin{theorem}\label{statptsandhetorbs}
For a barotropic van der Waals type pressure law $\hat p$, there is a non-trivial one-parameter family 
\be\label{endstates}
\left((\underline\nu,\underline U^0,\underline U^1)(m),
(\overline\nu,\overline U^0,\overline U^1)(m)\right), \quad
0\le m<\bar m,
\ee
of pairs of non-identical states such that for every value of the parameter $m$, system 
\eqref{EKbaro}--\eqref{defcheckK} has two heteroclinic solutions 
$$
\begin{aligned}
&x^1\mapsto (\overrightarrow \nu_m(x^1),\overrightarrow U^0_m(x^1),\overrightarrow U^1_m(x^1),0,0)),\\
&x^1\mapsto (\overleftarrow  \nu_m(x^1),\overleftarrow  U^0_m(x^1),\overleftarrow  U^1_m(x^1),0,0)),
\end{aligned}
$$
with
\be
\begin{aligned}
(\overrightarrow \nu_m,\overrightarrow U^0_m,\overrightarrow U^1_m)(-\infty)&=(\underline \nu,\underline U^0,\underline U^1)(m)\\
(\overrightarrow \nu_m,\overrightarrow U^0_m,\overrightarrow U^1_m)(+\infty)&=(\overline  \nu,\overline U^0,\overline U^1)(m)
\end{aligned}
\ee
and
\be
\begin{aligned}
(\overleftarrow \nu_m,\overleftarrow U^0_m,\overleftarrow U^1_m)(-\infty)&=(\overline \nu,\overline  U^0,\overline  U^1)(m)\\
(\overleftarrow \nu_m,\overleftarrow U^0_m,\overleftarrow U^1_m)(+\infty)&=(\underline\nu,\underline U^0,\underline U^1)(m).
\end{aligned}
\ee
\end{theorem}
A pressure law $$\check p(\nu)=\hat p(r(\nu))=\nu r'(\nu)-r(\nu)$$ 
is of van der Waals type if 
there exist two values $\nu_B>\nu_A>0$ such that 
\be\label{vdW}
\begin{aligned}
r''(\nu)>0\quad&\text{for}\quad 0<\nu<\nu_A\text{ or }\nu>\nu_B\\
r''(\nu)<0\quad&\text{for}\qquad\nu_A<\nu<\nu_B.
\end{aligned}
\ee
We seek for the pairs \eqref{endstates} as solving both the Rankine-Hugoniot conditions
\be\label{RHem}
\begin{aligned}
(\rho+p)U^0U^1&=q^0\\
(\rho+p)U_1^2+p&=q^1 
\end{aligned}
\ee
associated with the Euler equations \eqref{Eulerbaro}  
and the analogous relation
\be\label{RHnu}
\nu U^1=m
\ee
associated with the additional conservation law \eqref{nuconservation}, 
for appropriate values of the parameters $q^0,q^1,m$ with $q^1>0$ and, as amounts to restricting 
attention to $U^1\ge 0$, with $q^0,m\ge 0$. Using \eqref{RHnu}, we rewrite \eqref{RHem} as
\be\label{RHred}
\begin{aligned}
(r'(\nu))^2\left(1+\frac{m^2}{\nu^2}\right)&=c^2\\
\check p(\nu)+\frac{r'(\nu)}\nu m^2&=q^1 
\end{aligned}
\ee
with 
$$
\text{$c=q^0/m$ if $m,q^0>0$, and 
free otherwise.}
$$ 
\begin{lemma}[Maxwell states]\label{Maxwell} For $m=0$ there exists exactly one value $(c,q^1)\in(0,\infty)^2$ 
for which \eqref{RHred} has a (unique) solution $\nu_0^-\in(0,\nu_A)$ and a (unique) solution 
$\nu_0^+\in(\nu_B,\infty)$.
 \end{lemma}
\begin{proof}
For $m=0$ and momentarily writing $\pi$ for $q^1$, \eqref{RHred} reads 
\be\label{RHnu0}
\begin{aligned}
r'(\nu)&=c\\
\check p(\nu)&=\pi. 
\end{aligned}
\ee
Equation \eqref{RHnu0}$_2$ has unique solutions 
$$
\nu_-(\pi)\in (0,\nu_A),\quad \nu_+(\pi)\in(\nu_B,\infty)
$$
iff $\check p(\nu_B)<\pi<\check p(\nu_A)$.
As the quantity 
$$
I(\pi)=r'(\nu_+(\pi))-r'(\nu_-(\pi))=\int_{\nu_-(\pi)}^{\nu_+(\pi)}\frac{\check p(\nu)}\nu d\nu
$$
satisfies 
$$
I'(\pi)=\frac1{\nu_+(\pi)}-\frac1{\nu_-(\pi)}<0
\quad\text{and}\quad
I(\check p(\nu_B))>0>I(\check p(\nu_A)),
$$
there exists a unique $\pi_*$ with $I(\pi_*)=0$. This shows the assertion with $q^1=\pi_*, 
c=r'(\nu_0^-)=r'(\nu_0^+)$ and $\nu_0^-=\nu_-(\pi_*),\nu_0^+=\nu_+(\pi_*)$.
\end{proof}
Consider now, in principle for any real $m$, the curve $X_m:(0,\infty)\to\R^2$,
$$
X_m(\nu)=\begin{pmatrix}
         (r'(\nu))^2\left(1+\displaystyle{\frac{m^2}{\nu^2}}\right)\\ 
         \check p(\nu)+\displaystyle{\frac{r'(\nu)}\nu m^2}
         \end{pmatrix}
$$
associated with \eqref{RHred}.
\begin{lemma}\label{nuofm}
For $m\ge 0$ sufficiently small, there exist unique continuous functions $m\mapsto \nu^\pm(m)$ such that
(i) $X_m$ intersects itself transversely at $X_m(\nu^-(m))=X_m(\nu^+(m))$, and
(ii) $\nu^-(0)=\nu_0^-$, $\nu^+(0)=\nu_0^+$.
\end{lemma}
\begin{proof}
By continuity, it suffices to show that $X_0$ intersects itself transversely at 
\be\label{X00}
X(\nu_0^-)=X_0(\nu_0^+).
\ee
However, \eqref{X00} holds indeed as these two points are both equal to the value $(c^2,q^1)$ found in 
Lemma \ref{Maxwell}. And since
$$
X_0'(\nu)=\begin{pmatrix}
            2r'(\nu)r''(\nu)\\ \nu r''(\nu) 
            \end{pmatrix}
=r''(\nu)\begin{pmatrix}
         2r'(\nu)\\ \nu 
         \end{pmatrix}
$$
as well as
$$
r'(\nu_0^-)=r'(\nu_0^+)\quad\text{and}\quad r''(\nu_0^\pm)\neq 0,
$$
the tangent directions $X_0'(\nu_0^-)$ and $X_0'(\nu_0^+)$ are not parallel.
\end{proof}
\noindent
For $0\le m\le \bar m$, we define $c(m),q^1(m)>0$ by $((c(m))^2,q^1(m))=X_m(\nu^\pm(m))$.
\par\medskip
Turning to finding the traveling waves, we write 
$\dot{}$ for $\partial/\partial x^1$ and briefly 
$$
\bar\rho(\nu,\dot\nu)=r(\nu)+\frac12\kappa(\nu)\dot\nu^2,
\quad 
\bar f(\nu,\dot\nu,\ddot\nu)
=
r'(\nu)-\frac{\kappa'(\nu) }2 \dot\nu^2
-\kappa(\nu)\ddot\nu.
$$
\begin{lemma}
(i) A function $x\mapsto(\nu(x^1),U^0(x^1),U^1(x^1),0,0)$ solves \eqref{EKbaro}--\eqref{defcheckK} iff
there exist $m,c\in\R$ such that $\nu$ satisfies
\be\label{profilebarf}
\bar f(\nu,\dot\nu,\ddot\nu)\left(1+\left(\frac m\nu\right)^2\right)^{1/2}=c.
\ee
with the same $m$ and some $c\in\R$, while
$U^1(x^1)=m/\nu(x^1)$ (and $U^0(x^1)=(1+(U^1(x^1))^2)^{1/2}$).\\
(ii) With 
\be
a_{mc}(\nu)=c\left(m^2+\nu^2\right)^{1/2}  
\ee
and written as a first-order system, \eqref{profilebarf} reads 
\be\label{fods}
\begin{aligned}
\dot\nu&=\omega\\
\kappa(\nu)\dot\omega&=r'(\nu)-\frac{\kappa'(\nu)}2\omega^2-a_{mc}'(\nu). 
\end{aligned}
\ee
and has the first integral
\be\label{J}
J_{mc}(\nu,\omega)=r(\nu)-\frac12\kappa(\nu)\omega^2-a_{mc}(\nu).
\ee
\end{lemma}
\begin{proof}
(i) In this context, the PDE \eqref{nuconservationagain} is equivalent to the algebraic relation 
\eqref{RHnu}, and then \eqref{eombarf} readily reduces to \eqref{profilebarf}. (ii) follows trivially.
\end{proof}
\noindent
We define $J_m$ as $J_{mc}$  and $j_m$ as $r-a_{mc}$, both with $c=c(m)$ from above, so that 
\be
J_m(\nu,\omega)=j_m(\nu)-\kappa(\nu)\omega^2.
\ee
\begin{lemma}
There are $\tilde m>0$ and an open interval $N\subset(0,\infty)$ such that for every $m\in[0,\tilde m]$, 
the potential $J_m$  
has exactly three stationary points, $(\nu^-(m),0),(\nu^0(m),0),(\nu^+(m),0)$ in $N\times\R$, with 
$\nu^-(m),\nu^+(m)$ from Lemma \ref{nuofm} and an appropriate $\nu^0(m)$ with 
\be\label{nununu}
\nu^-(m)<\nu^0(m)<\nu^+(m).
\ee
Of these, $(\nu^-(m),0)$ and $(\nu^+(m),0)$ 
are nondegenerate saddles, and $(\nu^0(m),0)$ is a nondegenerate maximum, such that 
$$
J(\nu^-(m),0)=J(\nu^+(m),0)<J(\nu^0(m),0).
$$
\end{lemma}
\begin{proof}
In view of the definition of $\nu^\pm(m)$, \eqref{RHred}$_2$ implies
$$
j_m(\nu^-(m))=j_m(\nu^+(m))=-q^1(m). 
$$
while \eqref{RHred}$_1$ entails 
$$
j_m'(\nu^-(m))=j_m'(\nu^+(m))
$$
As $\nu^-(m)<\nu_A$ and $\nu^+(m)>\nu_B$, we also have 
$$
j_m''(\nu^-(m)),j_m''(\nu^+(m))>0.
$$
The van der Waals property  \eqref{vdW} tells us that at least for sufficiently small $m$,
$j_m$ must have exactly one stationary point between its two strict minima at $\nu^\pm(m)$, 
namely a strict maximum; i.e., there is a unique $\nu^0(m)$ with \eqref{nununu} and
$$
j_m''(\nu^0(m))<0\quad\text{and}\quad j_m(\nu^0(m))>j_m(\nu^\pm(m)).
$$
The assertion follows as 
$$
DJ_m(\nu,\omega)=(0,0)\Leftrightarrow j_m'(\nu)=\omega=0
\quad\text{and } D^2J_m(\nu,\omega)=\begin{pmatrix}
                                    j_m''(\nu)&0\\ 0&-\kappa(\nu) 
                                    \end{pmatrix}.
$$ 
\end{proof}
The following obvious corollary readily implies the statement of Theorem \ref{statptsandhetorbs}.
\begin{lemma}
The level set $\{(\nu,\omega):J_m(\nu,\omega)=-q^1(m)\}$ contains two regular curves, symmetric to each 
other, connecting $(\nu^-(m),0)$ and $(\nu^+(m),0)$ in $\{\omega>0\}$ and in $\{\omega>0\}$, 
respectively.
\end{lemma}
We detail the point made at the end of Section 1:
\begin{remark}\label{Rk1}
While for non-relativistic \emph{barotropic} Korteweg fluids mass and momentum are the primary conserved 
quantities and it is the conservation of energy which as companion law yields the additional 
jump condition (``kinetic rule'') \cite{B1}, we see that for relativistic \emph{barotropic} Korteweg 
fluids energy and momentum are the primary conserved quantities and the conservation of \emph{entropy} 
plays the r\^ole of the companion law that provides the kinetic rule. 
\end{remark}

\section{Phase boundaries in non-barotropic Korteweg fluids}
The dynamics of a perfect non-barotropic fluid is governed -- cf.\ \cite{F18} -- by 
the non-barotropic Euler equations
\be\label{Eulernonbaro}
T^{\alpha\beta}_{,\beta}=0,\quad N^\beta_{,\beta}=0,
\ee
where in the stress-energy tensor and the particle current vector,
\be\label{TN}
T^{\alpha\beta}=\rho U^\alpha U^\beta+p\Pi^{\alpha\beta},\quad N^\beta=nU^\beta,
\ee
$n$ is the particle number density, and energy and pressure can be viewed as functions 
\be\label{deftilderhotildep}
\rho=\tilde\rho(n,s),\quad p=\tilde p(n,s)
\ee
of $n$ and the specific entropy density $s$; the two functions are related by
\be\label{reltilderhotildep} 
\tilde p(n,s)=n\tilde \rho(n,s)-\tilde\rho(n,s).
\ee
For smooth solutions, the resulting companion law $$(SU^\beta)_{,\beta}=0$$ for the entropy 
$S=ns$ combines with the particle number conservation law \eqref{Eulernonbaro}$_2$ to
$$
U^\beta s_{,\beta}=0,
$$
i.\ e., constancy of specific entropy along particle paths. This property allows for restricting 
attention to \emph{isentropic solutions}, i.e., solutions for which the specific entropy is everywhere 
the same,
$$
s=s_*;
$$
for any given value of $s_*$, the system \eqref{Eulernonbaro} is identical with the description 
\eqref{Eulerbaro}, \eqref{idealstress}, \eqref{nconservation} of the isentropic fluid given by  
\eqref{rhoofn} with 
\be\label{ristilderho}
r(n) =\tilde\rho(n,s_*).
\ee
We again work with the extended energy and extended pressure,
now 
\be\label{barrhons}
\bar\rho(n,\partial n,s)=\tilde\rho(n,s)+\frac12\kappa(n,s)n_{,\gamma}n^{,\gamma}
\ee
and
\be\label{barpns}
\bar p(n\partial n,\partial^2 n,s)=n\frac{\delta\bar\rho(n,\partial n,s)}{\delta n}-\bar\rho(n\partial n,s).
\ee  
We define the Euler-Korteweg equation of such non-barotropic fluid as 
\be\label{EulerKortewegnonbaro}
\bar T^{\alpha\beta}_{,\beta}=0,\quad N^\beta_{,\beta}=0,
\ee
with \eqref{fullT}, \eqref{barrhons}, \eqref{barpns}, 
\be\label{defcheckKof}
\check K^{\alpha\beta}
=
\kappa n^{,\alpha} n^{,\beta}
\ee
the non-scalar part of 
\be\label{defKofn}
K^{\alpha\beta}
=
-n(\kappa n^{,\gamma})_{,\gamma}\Pi^{\alpha\beta}
+\kappa n^{,\alpha}n^{,\beta},
\ee
and \eqref{TN}$_2$.
\par\medskip
We call the fluid \eqref{deftilderhotildep}, \eqref{reltilderhotildep} of van der Waals type 
at a given value $s_*$ of specific entropy if, in specialization of \eqref{nonconvex}, the second 
derivative $r''(n)$ of $r$ from \eqref{ristilderho} changes sign first from positive to negative and 
then from negative to positive, as $n$ varies from $0$  to $\infty$. 
\par\medskip

Applying Theorem \ref{statptsandhetorbs} to this situation readily yields 
\begin{theorem}\label{EPBnonbaro}
Consider a non-barotropic fluid  \eqref{deftilderhotildep}, \eqref{reltilderhotildep} that is 
of van der Waals type at a given value $s_*$ of the specific entropy. Then there
is a non-trivial one-parameter family 
\be\label{nendstates}
\left((\underline n,\underline U^0,\underline U^1)(m),
(\overline n,\overline U^0,\overline U^1)(m)\right), \quad
0\le m<\bar m,
\ee
of pairs of non-identical states such that for every value of the parameter $m$, system 
\eqref{EulerKortewegnonbaro}
with \eqref{fullT}, \eqref{barrhons}, \eqref{barpns}, \eqref{defcheckKof}
has two heteroclinic solutions 
$$
\begin{aligned}
&x^1\mapsto (\overrightarrow n_m(x^1),s_*,\overrightarrow U^0_m(x^1),\overrightarrow U^1_m(x^1),0,0)),\\
&x^1\mapsto (\overleftarrow  n_m(x^1),s_*,\overleftarrow  U^0_m(x^1),\overleftarrow  U^1_m(x^1),0,0)),
\end{aligned}
$$
with
\be
\begin{aligned}
(\overrightarrow n_m,\overrightarrow U^0_m,\overrightarrow U^1_m)(-\infty)&=(\underline n,\underline U^0,\underline U^1)(m)\\
(\overrightarrow n_m,\overrightarrow U^0_m,\overrightarrow U^1_m)(+\infty)&=(\overline  n,\overline U^0,\overline U^1)(m)
\end{aligned}
\ee
and
\be
\begin{aligned}
(\overleftarrow n_m,\overleftarrow U^0_m,\overleftarrow U^1_m)(-\infty)&=(\overline n,\overline  U^0,\overline  U^1)(m)\\
(\overleftarrow n_m,\overleftarrow U^0_m,\overleftarrow U^1_m)(+\infty)&=(\underline n,\underline U^0,\underline U^1)(m).
\end{aligned}
\ee
\end{theorem}
We conclude by statin the counterpart of Remark \ref{Rk1} for isentropic 
\begin{remark}
While for non-relativistic \emph{isentropic} Korteweg fluids mass and momentum are the primary conserved 
quantities and it is the conservation of energy which as companion law yields the additional 
jump condition (``kinetic rule'') \cite{B1}, we see that for relativistic \emph{isentropic} Korteweg 
fluids energy and momentum are the primary conserved quantities and the conservation of \emph{mass} 
plays the r\^ole of the companion law that provides the kinetic rule. 
\end{remark}

\end{document}